\documentclass[reqno]{amsart}

\usepackage{enumerate}
\usepackage{tabto}
\usepackage[mathscr]{euscript}
\usepackage{xcolor}
\usepackage{layout}
\usepackage{fancyhdr}
\usepackage{array}
\usepackage{amsfonts}
\usepackage{amsmath}
\usepackage{amssymb}
\usepackage{mathtools}
\usepackage{graphicx}
\usepackage{bm}
\usepackage{enumitem}
\usepackage{caption} 
\usepackage{color}
\usepackage{kotex}
\usepackage{csquotes}
\usepackage{bookmark}
\usepackage{float}
\usepackage{multirow}
\usepackage[square,numbers,sort&compress]{natbib}
\usepackage{hyperref}
\hypersetup{colorlinks=true,linkcolor=blue,citecolor=red}
\allowdisplaybreaks

\def\XXint#1#2#3{{\setbox0=\hbox{$#1{#2#3}{\int}$ }
\vcenter{\hbox{$#2#3$ }}\kern-.6\wd0}}

\newtheorem{theorem}{Theorem}[section]

\newtheorem{proposition}[theorem]{Proposition}
\newtheorem{remark}[theorem]{Remark}
\theoremstyle{definition}
\newtheorem{definition}[theorem]{Definition}

\numberwithin{equation}{section}

\newcommand{ \mc }{ \mathcal }
\newcommand{ \mr }{ \mathbb{R} }

\newcommand{\iints}[1]{{\int\hspace{-0.28cm}\int_{#1}}}

\newcommand{ \dist }{ \operatorname{dist} }
\newcommand{ \supp }{ \operatorname{supp} }
\newcommand{ \loc }{ \operatorname{loc} }

\begin{document}
\title[Absence of the Lavrentiev phenomenon]{Absence of the Lavrentiev phenomenon for degenerate parabolic double phase problems}

\author{Bogi Kim}\address{Department of Mathematics, Kyungpook National University, Daegu, 41566, Republic of Korea} \email{rlaqhrl4@knu.ac.kr} \author{Youngchae Kim}\address{Department of Mathematics, Kyungpook National University, Daegu, 41566, Republic of Korea} \email{rjseka5241@knu.ac.kr} \author{Jehan Oh}\address{Department of Mathematics, Kyungpook National University, Daegu, 41566, Republic of Korea} \email{jehan.oh@knu.ac.kr}

\date{\today.}
\thanks{Bogi Kim is supported by the National Research Foundation of Korea (NRF) grant funded by the Korea government [Grant No. RS-2025-25426375].
Youngchae Kim is supported by the National Research Foundation of Korea (NRF) grant funded by the Korea government [Grant No. RS-2025-25433309].
Jehan Oh is supported by the National Research Foundation of Korea (NRF) grant funded by the Korea government [Grant Nos. RS-2025-00555316, and RS-2025-25415411].}

\begin{abstract}
We establish the absence of the Lavrentiev phenomenon for degenerate parabolic double phase problems. Any finite-energy function in the natural parabolic class admits smooth approximations with convergence in the parabolic Sobolev space and convergence of the corresponding energy. We provide explicit gap bound conditions and derive improved bounds under additional assumptions such as boundedness or stronger time regularity.
\end{abstract}
\maketitle

\section{\bf Introduction}\label{section 1}
In this paper, we introduce a new parabolic double phase functional and establish the absence of the Lavrentiev phenomenon for this class of functionals. Before describing the new functional framework, we briefly recall the concept of the Lavrentiev phenomenon in variational analysis. One of the fundamental questions in the calculus of variations concerns whether minimizers exist and whether the infimum is stable under approximation by smoother or more regular functions. The Lavrentiev phenomenon, however, presents a counter-intuitive case that challenges this expectation. Specifically, for a functional $\mathcal{G}$ defined on a function space $X$, its infimum over $X$ may be strictly smaller than the infimum taken over a dense subspace $Y\subset X$:
$$
\inf_X \mathcal{G}<\inf_Y \mathcal{G}.
$$
In particular, for non-autonomous functionals of the type
$$
u\mapsto \mathcal{G}_0(u,\Omega):= \int_\Omega G(x,Du)\, dx,
$$
where $\Omega$ is an open subset of $\mr^n$ ($n\geq 2$), the Lavrentiev phenomenon occurs in a ball $B\subset\subset \Omega$ if
$$
\inf_{u\in u_0+W_0^{1,p}(B)} \mathcal{G}_0(u,B)<\inf_{u\in u_0+W^{1,p}_0(B)\cap W^{1,q}_{\operatorname{loc}}(B)}\mathcal{G}_0(u,B),
$$
 where $u_0$ is a boundary data function with suitable regularity and the Carath\'{e}odory function $G:\Omega\times\mr^n\rightarrow\mr$ satisfies non-standard growth conditions (the terminology is introduced in \cite{Marcellini1989, Marcellini1991,Marcellini2020}) of $(p,q)$-type
$$
|\xi|^p\lesssim G(x,\xi)\lesssim |\xi|^q+1, \qquad 1<p<q.
$$
However, this phenomenon should not be viewed merely as an obstruction to regularity. According to \cite{Mingione2021}, it is a crucial observation that the same assumptions ensuring a priori regularity estimates for the minimizers can also be employed to prove the absence of the Lavrentiev phenomenon. Furthermore, this relationship may hold in reverse: by first establishing the absence of the Lavrentiev phenomenon, one can leverage this result to derive regularity estimates for the solutions. In this context, the absence of the Lavrentiev phenomenon means that for a map $u\in W^{1,p}(B)$, there exists a sequence of more regular $u_k$ such that for any ball $B\subset\subset \Omega$,
$$
\mathcal{G}_0(u_k,B)\rightarrow\mathcal{G}_0(u,B).
$$

The elliptic double phase functional
$$
W^{1,1}(\Omega)\ni w \mapsto \mathcal{H}_0(w,\Omega):= \int_\Omega \left[|Dw|^p+a(x)|Dw|^q\right] dx
$$
with an open set $\Omega\subset\mr^n$, $1 < p < q$ and $0\leq a(\cdot)\in C^{\alpha}(\Omega)\; (\alpha\in (0,1])$ serves as one of the most prominent examples of the Lavrentiev phenomenon, see \cite{Zhikov1986,Zhikov1993,Zhikov1995,Zhikov1997}. The corresponding Euler-Lagrange equation is given by
$$
-\operatorname{div}(p|Du|^{p-2}Du + q a(x)|Du|^{q-2}Du)=0,
$$
which characterizes the homogenization of strongly anisotropic materials. To resolve the Lavrentiev phenomenon for this problem, we require a gap bound condition that relates the proximity of $p$ and $q$ and the regularity of $a(\cdot)$. In fact, Esposito-Leonetti-Mingione \cite{Esposito2004} and Colombo-Mingione \cite{Colombo2015} have established that the gap bound condition 
\begin{equation}\label{cond : gap bound condition of elliptic double phase}
\frac{q}{p} \leq 1 + \frac{\alpha}{n}
\end{equation}
is a necessary and sufficient condition for the absence of the Lavrentiev phenomenon. More precisely, this condition implies that the gradient of a weak solution $u$ is in $L^q(\Omega)$. Alternatively, the gap bound condition can be relaxed by imposing an a priori condition on $u$. For example, the absence of the Lavrentiev phenomenon is also guaranteed under the condition
\begin{equation}\label{cond : gap bound condition with bounded solution of elliptic double phase}
u\in L^\infty(\Omega)\quad \text{and}\quad q\leq p +\alpha,
\end{equation}
see \cite{Fonseca2004,Esposito2004,Colombo2015a}. Also, in \cite{Bulicek2022}, if
$$
u_0\in W^{1,q}(\Omega)\quad \text{and}\quad q\leq p+\alpha\max\left\{1,\frac{p}{n}\right\},
$$
the Lavrentiev phenomenon does not occur, that is, for any ball $B\subset\subset \Omega$,
$$
\inf_{u\in u_0+W^{1,p}_0 (B)} \mathcal{H}_0(u,B) = \inf_{u\in u_0+W^{1,q}_0(B)} \mathcal{H}_0(u,B)=\inf_{u\in u_0+C_c^\infty (B)}\mathcal{H}_0(u,B).
$$
On the other hand, in \cite{borowski2025}, it is demonstrated that the gap bound condition can be weakened while assuming higher regularity for $a(\cdot)$, still ensuring the absence of the Lavrentiev phenomenon:
$$
a\in C^{k,\alpha}(\Omega)\quad \text{and}\quad q\leq p+(k+\alpha)\max\left\{1,\frac{p}{n}\right\}.
$$
Imposing such gap bound conditions allows one to avoid the Lavrentiev phenomenon. Furthermore, these assumptions are also essential for improving the regularity of $u$. Under either \eqref{cond : gap bound condition of elliptic double phase} or \eqref{cond : gap bound condition with bounded solution of elliptic double phase}, Baroni-Colombo-Mingione \cite{Baroni2018} have proved that the gradient of a weak solution $u$ is locally H\"{o}lder continuous. Additionally, the authors have demonstrated that the same conclusion remains valid whenever
$$
u\in C^\gamma(\Omega)\quad \text{and}\quad q\leq p+\frac{\alpha}{1-\gamma}\quad \text{with }\gamma\in(0,1)
$$
is satisfied. Moreover, it is shown in \cite{Ok2020} that the local H\"{o}lder continuity of a local quasi-minimizer $u$ is also guaranteed under the assumption
$$
u\in L^\gamma_{\operatorname{loc}}(\Omega)\quad \text{and}\quad q\leq p +\frac{\gamma\alpha}{n+\gamma}
$$
for $p\in(1,n)$ and $\gamma>\frac{np}{n-p}$. In addition, regarding weak solutions, Baroni-Colombo-Mingione \cite{Baroni2015} and Ok \cite{Ok2017} have proved Harnack's inequality and H\"{o}lder continuity. Also, Calder\'{o}n-Zygmund estimates were derived in \cite{Baasandorj2020,Colombo2016,DeFilippis2019}. More broadly, a significant body of regularity results for elliptic double phase problems has been developed in \cite{Byun2021,Byun2021a,Byun2017,Byun2020,2023DeFilippis,Haestoe2022,Haestoe2022a,Kim2026,Kim2024a,Kim2025a}.

The model equation of parabolic double phase problems is given by
\begin{equation}\label{def : main model equation of parabolic double phase problem}
u_t-\operatorname{div}(|Du|^{p-2}Du+a(x,t)|Du|^{q-2}Du)=0\quad \text{in }\Omega_T:= \Omega\times (0,T),
\end{equation}
where $n\geq 2$, $\frac{2n}{n+2}<p<q<\infty$, $T>0$ and $\Omega$ is a bounded open set in $\mr^n$. Here, we assume that
\begin{equation}\label{cond : assumption of w and a}
0\leq a(\cdot)\in C^{\alpha,\frac{\alpha}{2}}(\Omega_T)\quad (\alpha\in(0,1]).
\end{equation}
Here, $a(\cdot)\in C^{\alpha,\frac{\alpha}{2}}(\Omega_T)$ means that $a(\cdot)\in L^\infty(\Omega_T)$ and there exists a H\"{o}lder constant $[a]_\alpha:= [a]_{\alpha,\frac{\alpha}{2};\Omega_T}>0$ such that
$$
|a(x_1,t_1)-a(x_2,t_2)|\leq [a]_\alpha \max\{|x_1-x_2|^\alpha,|t_1-t_2|^{\frac{\alpha}{2}}\}
$$
for all $x_1,\,x_2\in\Omega$ and $t_1,\,t_2\in(0,T)$. The definition of weak solutions for \eqref{def : main model equation of parabolic double phase problem} is given as follows:
\begin{definition}\label{def : weak solution}
    A function $u:\Omega_T\rightarrow\mr$ with
    $$
    u\in C(0,T;L^2(\Omega))\cap L^1(0,T;W^{1,1}(\Omega))
    $$
    and
    $$
    \iints{\Omega_T} H(z,|Du|)\, dz<\infty
    $$
    is a weak solution to \eqref{def : main model equation of parabolic double phase problem} if
    $$
    \iints{\Omega_T} \left[-u\cdot \varphi_t + (|Du|^{p-2}Du+a(z)|Du|^{q-2}Du)\cdot D\varphi\right] dz=0
    $$
    for every $\varphi\in C_0^\infty(\Omega_T)$.
\end{definition}
The Lavrentiev phenomenon for parabolic double phase problems has only recently been investigated in \cite{Chlebicka2019,Chlebicka2021}. In these papers, the absence of the Lavrentiev phenomenon has been proved for a locally integrable $N$-function under a suitable balance condition, which encompasses the parabolic double phase problems as an example. In addition, similar to the results for the elliptic double phase problem, the gap bound condition guarantees the regularity of the weak solutions. Indeed, for the degenerate parabolic double phase problems (that is, $p\geq 2$), under the gap bound condition
$$
q\leq p+\frac{2\alpha}{n+2},
$$ 
the energy estimates of weak solutions and the existence theory of weak solutions have been established in Kim-Kinnunen-S\"{a}rki\"{o} \cite{Wontae2023a} (see also \cite{Chlebicka2019,Singer2016}) and the (spatial) gradient higher integrability results have been studied in Kim-Kinnunen-Moring \cite{2023_Gradient_Higher_Integrability_for_Degenerate_Parabolic_Double-Phase_Systems}.
Furthermore, as in the elliptic case, improved regularity of $u$ allows for a relaxation of the gap bound condition. Indeed, Chlebicka-Garain-Kim \cite{chlebicka2025gradient} and Kim-Oh \cite{kim2025bounded} have obtained gradient higher integrability under the assumption
$$
u\in L^\infty(\Omega_T)\quad \text{and}\quad q\leq p +\alpha.
$$
Also, in \cite{kim2025bounded}, the authors demonstrated an interpolation of the gap bound condition for the degenerate parabolic double phase problem by showing that higher integrability results hold even under the assumption
$$
u\in C(0,T;L^s(\Omega))\quad \text{and}\quad q\leq p +\frac{s\alpha}{n+s}.
$$
These results hold similarly for singular parabolic double phase problems (i.e., $\frac{2n}{n+2}<p\leq 2$); for further details, refer to \cite{Wontae2024,kim2026interpolative}. Beyond these, regularity results concerning the parabolic double phase problems can be found in \cite{Buryachenko2022,Wontae2025,Kim2024,Kim2025,Sen2025,Wontae2023b}.

Now, in order to prove the absence of the Lavrentiev phenomenon, we propose a new parabolic double phase functional. Regarding the elliptic setting, focusing on local minimizers rather than weak solutions naturally allowed for the study of the associated double phase functional, which in turn served as a key tool for proving the absence of the Lavrentiev phenomenon. Accordingly, to address the absence of the Lavrentiev phenomenon in the parabolic setting, a functional associated with the parabolic double phase is required. While one may refer to \cite{Kinnunen2015,Habermann2016,Habermann2015,Fujishima2014} for the theory of parabolic quasi-minimizers, those frameworks are insufficient for establishing our desired results. Hence, we introduce a new functional in this paper. We consider the parabolic double phase functional
\begin{equation}\label{def : parabolic double phase functional}
\mc{F}(w,\Omega_T):=\sup_{t\in[0,T]}\int_{\Omega}|w(x,t)|^2\,dx+\mc{P}(w,\Omega_T),
\end{equation}
where $\Omega_T$ is a parabolic cylinder with an open bounded set $\Omega\subset \mr^n$, $n\geq 2$ and $2\leq p < q$. Here, $\mathcal{P}(w,\Omega_T)$ is the parabolic double phase functional without $L^2$-energy term given by
\begin{equation}    \label{def : stationary part of the parabolic double phase functional}
\mathcal{P}(w,\Omega_T):= \iints{\Omega_T} \left[ |Dw(x,t)|^p + a(x,t)|Dw(x,t)|^q\right] dxdt.
\end{equation}
\begin{remark}
    The design of this functional is motivated by the function space described in Definition \ref{def : weak solution}. The condition that a weak solution $u\in L^1(0,T;W^{1,1}(\Omega))$ satisfies 
    $$
    \iints{\Omega_T} H(z,|Du|)\, dz < \infty
    $$ 
    means that $\mathcal{P}(u,\Omega_T)$ is finite, while the $L^2$-energy term in \eqref{def : parabolic double phase functional} reflects the requirement $u \in C(0,T;L^2(\Omega))$. Therefore, the finiteness of $\mathcal{F}$ in \eqref{def : parabolic double phase functional} effectively implies that $u$ belongs to the space prescribed in Definition \ref{def : weak solution}.
\end{remark}

In this paper, we prove the absence of the Lavrentiev phenomenon by utilizing the functional defined in \eqref{def : parabolic double phase functional}. For this, we define $H:\Omega_T\times\mr^n\to[0,\infty)$ by
$$
H(z,\xi):=|\xi|^p+a(z)|\xi|^q
$$
for $z\in\Omega_T$ and $\xi\in\mr^n$. With a slight abuse of notation, we shall denote $H(z,\xi)$ also when $\xi\in[0,\infty)$. Also, we write a cylinder $Q$ as
$$
Q:= B\times I,
$$
where $B$ is a ball in $\mr^n$ and $I$ is an interval. The following theorem states the results concerning the absence of the Lavrentiev phenomenon for the parabolic double phase functional without the $L^2$-energy term under several gap bound conditions.

\begin{theorem}    \label{thm: main theorem 1}
    \rm Let $\mathcal{P}$ be the functional defined in \eqref{def : stationary part of the parabolic double phase functional} under \eqref{cond : assumption of w and a}. Then the following results hold:
    \begin{enumerate}[label = (\roman*)]
        \item\label{case : general absence of Lavrentiev phenomenon} If 
        \begin{equation}    \label{cond : general gap bound condition}
            q\leq p+\frac{p\alpha}{n+2},
        \end{equation}
        then for every function $w\in L^p_{\operatorname{loc}}(0,T;W^{1,p}_{\operatorname{loc}}(\Omega))$ and cylinders $Q\subset\subset\tilde{Q}\subset\subset\Omega_T$ satisfying $\mathcal{P}(w,\tilde{Q})<\infty$, there exists a sequence $\{w_m\}_{m=1}^\infty$ of $C^\infty$-functions such that
        $$
        w_m\rightarrow w \ \text{ in }L^p(I;W^{1,p}(B))\quad\text{and}\quad \mathcal{P}(w_m,Q)\rightarrow\mathcal{P}(w,Q).
        $$
        \item\label{case : absence of Lavrentiev phenomenon of bounded solution} The same conclusion holds for $w\in L^\infty_{\operatorname{loc}}(\Omega_T)\cap  L^p_{\operatorname{loc}}(0,T;W^{1,p}_{\operatorname{loc}}(\Omega))$ if
        \begin{equation}    \label{cond : gap bound condition of bounded solution}
            q\leq p+\max\left\{\frac{p\alpha}{n+2},\alpha\right\}.
        \end{equation}
        \item\label{case : absence of Lavrentiev phenomenon of s-condition} The same assertion remains valid for $w\in C_{\operatorname{loc}}(0,T;L^s_{\operatorname{loc}}(\Omega)) \cap L^p_{\operatorname{loc}}(0,T;$ $W^{1,p}_{\operatorname{loc}}(\Omega))$ with $s\geq 2$ if
        \begin{equation}    \label{cond : gap bound condtion of s-condition}
            q\leq p+\max\left\{\frac{s\alpha}{n+s},\frac{p\alpha}{n+2}\right\}.
        \end{equation}
    \end{enumerate}
\end{theorem}

The results regarding the absence of the Lavrentiev phenomenon for the parabolic double phase functional defined in \eqref{def : parabolic double phase functional} are stated as follows:
\begin{theorem} \label{thm: main theorem 2}
    \rm  Let $\mathcal{F}$ be the functional defined in \eqref{def : parabolic double phase functional} under \eqref{cond : assumption of w and a} and assume that
    $$
    w\in C_{\loc}(0,T;L_{\loc}^2(\Omega))\cap L^p_{\operatorname{loc}}(0,T;W^{1,p}_{\operatorname{loc}}(\Omega)).
    $$
    \begin{enumerate}[label = (\roman*)]
        \item If \eqref{cond : general gap bound condition} holds, then for every pair of cylinders $Q\subset\subset\tilde{Q}\subset\subset\Omega_T$ satisfying $\mathcal{F}(w,\tilde{Q})<\infty$, there exists a sequence $\{w_m\}_{m=1}^\infty$ of $C^\infty$-functions such that
        $$
        w_m\rightarrow w \ \text{ in }L^p(I;W^{1,p}(B)),\quad w_m\rightarrow w \ \text{ in }C(I;L^2(B))\quad
        $$
        $$
        \text{and}\quad \mathcal{F}(w_m,Q)\rightarrow\mathcal{F}(w,Q).
        $$
        \item The same conclusion holds if $w\in L_{\loc}^\infty(\Omega_T)$ is further assumed and \eqref{cond : gap bound condition of bounded solution} holds.
        \item The same result is obtained for
        $$
        w\in C_{\loc}(0,T;L_{\loc}^s(\Omega))\cap L^p_{\operatorname{loc}}(0,T;W^{1,p}_{\operatorname{loc}}(\Omega))\quad(s\geq 2)
        $$
        if \eqref{cond : gap bound condtion of s-condition} holds.
    \end{enumerate}
\end{theorem}

These theorems are proved by employing a simultaneous space-time mollification technique, as introduced in \cite{hasto2025higher}. This approach differs from that of \cite{Chlebicka2019,Chlebicka2021}, where the absence of the Lavrentiev phenomenon was established by mollifying the time and space variables separately. The rest of the paper is structured as follows:
\begin{itemize}
    \item Section \ref{section 2} introduces the parabolic mollification used throughout this work and explores its basic properties.
    \item Section \ref{section 3} provides the proof for the absence of the Lavrentiev phenomenon for the parabolic double phase functional without the $L^2$-energy term.
    \item Section \ref{section 4} focuses on establishing the $C_{\operatorname{loc}}(0,T;L_{\operatorname{loc}}^2(\Omega))$ convergence in order to prove Theorem \ref{thm: main theorem 2}.
\end{itemize}

\section{\bf Parabolic mollification}\label{section 2}
Define 
$$
\kappa(z)=\kappa(x,t):= \tilde{\kappa}_n(|x|)\tilde{\kappa}_1(t),
$$
where $\tilde{\kappa}_m\in C^\infty (\mr)$ is given by
$$
\tilde{\kappa}_m(s):= \begin{cases}
    c_m\operatorname{exp}\left(\frac{1}{s^2-1}\right)\quad &\text{if} \ \, |s|<1,\\
    0\quad &\text{if} \ \, |s|\geq 1,
\end{cases}
$$
with $c_m>0$ determined by $\int_{\mr^m}\tilde{\kappa}_m(|x|)\, dx =1$. Denote
$$
\kappa_h(x,t):= \frac{1}{h^{n+2}}\kappa\left(\frac{x}{h},\frac{t}{h^2}\right)
$$
for $h>0$. We notice that $\kappa_h\in C^\infty(\mr^{n+1})$, $\kappa_h\geq 0$, $\|\kappa_h\|_{L^1(\mr^{n+1})}=1$ and $\operatorname{supp}(\kappa_h)\subset Q_h(0)$, where $Q_r(z)=B_r(x)\times I_{r^2}(t)$ with $r>0$ and $z=(x,t)\in\mr^n\times\mr$. Moreover, $0\leq \kappa_h\leq c(n)h^{-n-2}$, $0\leq |D\kappa_h|\leq c(n)h^{-n-3}$ and $0\leq |\partial_t\kappa_h|\leq c(n)h^{-n-4}$.

We further define for $f\in L^1(U)$, where $U$ is a bounded open set in $\mr^{n+1}$,
$$
[f]^h (z):= (f * \kappa_h)(z)=\iints{\mr^{n+1}} f(z-\sigma)\kappa_h(\sigma)\,d\sigma=\iints{Q_h(0)} f(z-\sigma)\kappa_h(\sigma)\, d\sigma,
$$
where $Q_h(z)\subset U$.
\begin{proposition}
    Let $U'\subset\subset U$ and $0<h<2^{-1}\dist(\partial U,U')$. Then the following three statements hold:
    \begin{enumerate}[label=\upshape(\roman*)]
        \item $[f]^h\in C^\infty (U')$ with $D^\alpha [f]^h=f*D^\alpha\kappa_h$;
        \item $[f]^h\to f$ a.e. in $U$;
        \item if $f\in L^p(I;W^{k,p}(B))$ with $U'=Q$ is further assumed, then $[D^\alpha f]^h= D^\alpha[f]^h$ on $Q$ for every multi-index $|\alpha|\leq k$.
    \end{enumerate}
\end{proposition}
\begin{proof}
    (i) Note that $Q_h(z)\subset\subset U$ for any $z\in U'$. Choose any $\epsilon>0$ so small that $z+\epsilon e_i\in U'$ and $Q_h(z+\epsilon e_i)\subset\subset U$. Then we take an open set $V\subset\subset U$ so that
    $$
    Q_h(z)\cup Q_h(z+\epsilon e_i)\subset V\subset U.
    $$
    Since $\supp\kappa_h\subset Q_h(0)$, we obtain
    $$
    \begin{aligned}
        \frac{[f]^h(z+\epsilon e_i)-[f]^h(z)}{\epsilon} & =\iints{U} f(\sigma)\left( \frac{\kappa_h(z+\epsilon e_i-\sigma)-\kappa_h(z-\sigma)}{\epsilon} \right)\,d\sigma \\
        & =\iints{V} f(\sigma)\left( \frac{\kappa_h(z+\epsilon e_i-\sigma)-\kappa_h(z-\sigma)}{\epsilon} \right)\,d\sigma.
    \end{aligned}
    $$
    By the Mean Value Theorem, 
    $$
    \frac{\kappa_h(z+\epsilon e_i-\sigma)-\kappa_h(z-\sigma)}{\epsilon}\rightarrow D_{x_i}\kappa_h(z-\sigma)
    $$
    uniformly on $V$ as $\epsilon\rightarrow 0$. It follows that
    $$
    D_{x_i}[f]^h(z)=(f*D_{x_i}\kappa_h)(z).
    $$
    Similarly, we have
    $$
    D^\alpha[f]^h=f*D^\alpha\kappa_h\quad\text{on }U'.
    $$
    
    (ii) Let $z\in U$ be a Lebesgue point of $f$. Then we see that
    $$
    \begin{aligned}
        |[f]^h(z)-f(z)| & =\left| \iints{Q_h(z)}\kappa_h(z-\sigma)f(\sigma)\,d\sigma-\iints{Q_h(z)}\kappa_h(z-\sigma)f(z)\,d\sigma \right| \\
        & \leq\iints{Q_h(z)}\left| \kappa_h(z-\sigma) \right|\left| (f(\sigma)-f(z)) \right|\,d\sigma \\
        & \leq\frac{c(n)}{h^{n+2}}\iints{Q_h(z)}\left| f(\sigma)-f(z) \right|\,d\sigma \\
        & =\frac{c(n)}{|Q_h(z)|}\iints{Q_h(z)}|f(\sigma)-f(z)|\,d\sigma\to 0\quad\text{as }h\to 0.
    \end{aligned}
    $$
    
    (iii) Finally, suppose that $U'=Q$, $f\in L^p(I;W^{k,p}(B))$ and $z=(x,t)\in Q$. Using (i), we have
    $$
    \begin{aligned}
        D^\alpha[f]^h(z) & =\iints{Q_h(z)}f(\sigma)D_x^\alpha\kappa_h(z-\sigma)\,d\sigma \\
        & =\int_{I_{h^2}(t)}\int_{B_h(x)}f(y,\tau)D_x^\alpha\kappa_h(x-y,t-\tau)\,dyd\tau \\
        & =(-1)^{|\alpha|}\int_{I_{h^2}(t)}\int_{B_h(x)}f(y,\tau)D_y^\alpha\kappa_h(x-y,t-\tau)\,dyd\tau \\
        & =(-1)^{2|\alpha|}\int_{I_{h^2}(t)}\int_{B_h(x)}D_y^\alpha f(y,\tau)\kappa_h(x-y,t-\tau)\,dyd\tau \\
        & =\iints{Q_h(z)}D_y^\alpha f(\sigma)\kappa_h(z-\sigma)\,d\sigma = (D^\alpha f*\kappa_h)(z).
    \end{aligned}
    $$
\end{proof}

\begin{proposition}[\cite{hasto2025higher}, Proposition 4.1]    \label{prop: proposition 4.1 for 2025Hasto}
    Let $U\subset \mr^{n+1}$ be a bounded open set and let $\kappa_h\in C^\infty_c (\mr^{n+1})$ and $[f]^h$ be as above. Then, for $f\in L_{\operatorname{loc}}^H(U)$, $[f]^h\rightarrow f$ in $L_{\operatorname{loc}}^H(U)$ as $h\rightarrow0^+.$
\end{proposition}
Indeed, Proposition \ref{prop: proposition 4.1 for 2025Hasto} holds not only for $H$ but also for the functions $\varphi_t:\Omega_T\times[0,\infty)\to[0,\infty)$, $t\geq 1$, defined by $\varphi_t(z,s)=s^t$, see \cite{hasto2025higher}.

\section{\bf Lavrentiev phenomenon without $L^2$-energy term}\label{section 3}

\begin{proof}[Proof of Theorem \ref{thm: main theorem 1}]
Let $h_0>0$ be such that $Q+Q_{h_0}(0) := \{ z+w : z \in Q, w \in Q_{h_0}(0) \} \subset\subset\tilde{Q}=\tilde{B}\times\tilde{I}$ and denote $Q+Q_{h_0}(0)=Q_0=B_0\times I_0$. From now on, we always assume $0<h<h_0$ so that $[w]^h$ is well-defined in $Q$.

\ref{case : general absence of Lavrentiev phenomenon} Define
$$
[w]^h(z):=(w*\kappa_h)(z),\quad a_h(z):=\inf_{\zeta \in Q_h(z)}a(\zeta),
$$
and
$$
H_h(z,\xi):=|\xi|^p+a_h(z)|\xi|^q.
$$
for $z\in Q$ and $\xi\in\mr^n$. Since 
$$
w\in L_{\operatorname{loc}}^p(0,T; W_{\operatorname{loc}}^{1,p}(\Omega))
$$
and $\tilde{Q}\subset\subset\Omega_T$, we have
$$
\iints{\tilde{Q}} |w|^p+|Dw|^p\,dz = \int_{\tilde{I}}\|w\|_{W^{1,p}(\tilde{B})}^p\,dt=\|w\|_{L^p(\tilde{I};W^{1,p}(\tilde{B}))}^p<\infty.
$$
It follows that
$$
w,\,Dw\in L^{\varphi_p}(\tilde{Q})\subset L_{\operatorname{loc}}^{\varphi_p}(\tilde{Q}).
$$
By Proposition \ref{prop: proposition 4.1 for 2025Hasto}, $[w]^h\to w$ and $D[w]^h=[Dw]^h\to Dw$ in $L_{\operatorname{loc}}^{\varphi_p}(\tilde{Q})$. In particular,
$$
[w]^h\to w\quad\text{and}\quad D[w]^h\to Dw\quad\text{in }L^{\varphi_p}(Q).
$$
Therefore,
$$
\begin{aligned}
    \|[w]^h-w\|_{L^p(I;W^{1,p}(B))}^p & =\int_I\|[w]^h-w\|_{W^{1,p}(B)}^p\,dt \\
    & =\iints{Q}|[w]^h-w|^p\,dz+\iints{Q}|D[w]^h-Dw|^p\,dz\to 0
\end{aligned}
$$
as $h\searrow 0$; that is,
$$
[w]^h\to w\quad\text{in }L^p(I;W^{1,p}(B)).
$$

On the other hand, for every $z\in Q$ and $\xi\in\mr^n$,
\begin{equation}    \label{eq: estimate H by H_h}\tag{$\star$}
    H(z,\xi)=(a(z)-a_h(z))|\xi|^q+H_h(z,\xi)\leq ch^\alpha|\xi|^q+H_h(z,\xi)
\end{equation}
by the H\"{o}lder continuity for $a(\cdot)$. In particular, it holds for $\xi=D[w]^h(z)$. By Jensen's inequality with $\int_{\mr^{n+1}}\kappa_h\,d\sigma=1$, we have
\begin{equation}    \label{eq: estimate |D[w]^h|^p}\tag{$\star\star$}
    \begin{aligned}
        \left| D[w]^h(z) \right|^p & =\left| \iints{Q_h(0)}Dw(z-\sigma)\kappa_h(\sigma)\,d\sigma \right|^p \\
        & \leq\iints{Q_h(0)}|Dw(z-\sigma)|^p\kappa_h(\sigma)\,d\sigma\\
        &\leq ch^{-n-2}\iints{\tilde{Q}}|Dw|^p\, dz\leq ch^{-n-2}
    \end{aligned}
\end{equation}
for every $z\in Q$, where $c>1$ depends on $n$ and $\|H(\cdot, Dw)\|_{L^1(\tilde{Q})}$. Therefore,
$$
\begin{aligned}
    H(z,D[w]^h(z)) & \leq ch^\alpha|D[w]^h(z)|^q+H_h(z,D[w]^h(z)) \\
    & =ch^\alpha|D[w]^h(z)|^{q-p}|D[w]^h(z)|^p+H_h(z,D[w]^h(z)) \\
    & \leq ch^{\alpha-\frac{(n+2)(q-p)}{p}}H_h(z,D[w]^h(z))+H_h(z,D[w]^h(z)) \\
    & \leq cH_h(z,D[w]^h(z)).
\end{aligned}
$$
The last inequality holds as
\begin{equation}\label{cond : general assumption}
\alpha-\frac{(n+2)(q-p)}{p}\geq 0 \quad\Longleftrightarrow\quad q\leq p+\frac{p\alpha}{n+2}.
\end{equation}
By Jensen's inequality again, we have
$$
\begin{aligned}
    H_h(z,D[w]^h(z)) & \leq \iints{Q_h(0)}\left[ |Dw(z-\sigma)|^p+a_h(z)|Dw(z-\sigma)|^q \right]\kappa_h(\sigma)\,d\sigma \\
    & \leq \iints{Q_h(0)} \left[ |Dw(z-\sigma)|^p+a(z-\sigma)|Dw(z-\sigma)|^q \right]\kappa_h(\sigma)\,d\sigma \\
    & =[H(\cdot,Dw(\cdot))*\kappa_h](z)=[H(\cdot,Dw(\cdot))]^h(z).
\end{aligned}
$$
Thus,
$$
H(z,D[w]^h(z))\leq c[H(\cdot,Dw(\cdot))]^h(z)
$$
for every $z\in Q$. Since $\mathcal{P}(w,\tilde{Q})<\infty$, $H(\cdot,Dw(\cdot))\in L^{\varphi_1}(\tilde{Q})\subset L_{\operatorname{loc}}^{\varphi_1}(\tilde{Q})$. Applying Proposition \ref{prop: proposition 4.1 for 2025Hasto} again, we have $[H(\cdot,Dw(\cdot))]^h\to H(\cdot,Dw(\cdot))$ in $L^{\varphi_1}(Q)=L^1(Q)$. Since $D[w]^h\to Dw$ a.e. in Q, $H(z,D[w]^h (z))\to H(z,Dw(z))$ a.e. $z\in Q$. A generalization of the dominated convergence theorem completes our claim with $w_m:=[w]^{h_m}\in C^\infty(Q)$, where $h_m\searrow 0$.

\ref{case : absence of Lavrentiev phenomenon of bounded solution} As in part \eqref{eq: estimate |D[w]^h|^p}, we observe that for any $z\in Q$,
$$
\begin{aligned}
    \left| D[w]^h(z) \right| & =\left| \iints{Q_h(0)}w(z-\sigma)D\kappa_h(\sigma)\,d\sigma \right| \\
    & \leq \iints{Q_h(0)} \left| w(z-\sigma)D\kappa_h(\sigma) \right|\,d\sigma \\
    & \leq \|w\|_{L^\infty(\tilde{Q})}c(n)h^{-n-3}|Q_h| \\
    & \leq c h^{-1}.
\end{aligned}
$$
Thus, \eqref{eq: estimate H by H_h} implies 
$$
\begin{aligned}
    H(z,D[w]^h(z)) & \leq ch^\alpha|D[w]^h(z)|^q+H_h(z,D[w]^h(z)) \\
    & =ch^\alpha|D[w]^h(z)|^{q-p}|D[w]^h(z)|^p+H_h(z,D[w]^h(z)) \\
    & \leq ch^{\alpha-q+p}H_h(z,D[w]^h(z))+H_h(z,D[w]^h(z)) \\
    & \leq cH_h(z,D[w]^h(z))
\end{aligned}
$$
with the following assumption:
\begin{equation}\label{cond : bounded assumption}
\alpha-q+p\geq 0 \quad\Longleftrightarrow\quad q\leq p+\alpha.
\end{equation}
Since 
$n+2\geq p$ is equivalent to $\frac{p\alpha}{n+2}\leq\alpha$,
we have the desired result by arguing as in the case where \eqref{cond : bounded assumption} is assumed when $n+2\geq p$, and by following the approach used for \eqref{cond : general assumption} otherwise.

\ref{case : absence of Lavrentiev phenomenon of s-condition} As in the proof of \ref{case : absence of Lavrentiev phenomenon of bounded solution}, if
\begin{equation}\label{cond : s-assumption}
q\leq p +\frac{s\alpha}{n+s} \iff \alpha-\frac{(n+s)(q-p)}{s}\geq 0
\end{equation}
holds, then we obtain from H\"{o}lder's inequality that for any $z \in Q$,
$$
\begin{aligned}
    |D[w]^h(z)|& \leq \iints{Q_h(z)}|w(\sigma)||D\kappa_h(z-\sigma)|\,d\sigma\\
    &\leq ch^{-n-3}\int_{I_{h^2}(t)}\left(\int_{B_h(x)}|w(y,\tau)|^s\,dy\right)^\frac{1}{s}|B_h(x)|^{\frac{s-1}{s}}\,d\tau\\
    & \leq ch^{-3-\frac{n}{s}}\int_{I_{h^2}(t)}\|w\|_{C(I_0;L^s(B_0))} \,d\tau\\
    &\leq ch^{-\frac{n+s}{s}}.
\end{aligned}
$$
Therefore, for any $z\in Q$, we get
$$
\begin{aligned}
    H(z,D[w]^h(z)) & \leq ch^\alpha|D[w]^h(z)|^{q-p}|D[w]^h(z)|^p+H_h(z,D[w]^h(z)) \\
    & \leq ch^{\alpha-\frac{(n+s)(q-p)}{s}}H_h(z,D[w]^h(z))+H_h(z,D[w]^h(z)) \\
    & \leq cH_h(z,D[w]^h(z)).
\end{aligned}
$$

When $n+2 > p$ and $\frac{pn}{n-p+2}\leq s$, $\frac{p}{n+2}\leq \frac{s}{n+s}$. In this case, we proceed as in the setting where \eqref{cond : s-assumption} is imposed to obtain the desired result; otherwise, we adopt the approach used under \eqref{cond : general assumption}.
\end{proof}

\section{\bf Lavrentiev phenomenon with $L^2$-energy term}\label{section 4}

\begin{proof}[Proof of Theorem \ref{thm: main theorem 2}]
From Section \ref{section 3}, we already have the following: if at least one of our assumptions is true, then
$$
[w]^h\to w\quad\text{in }L^p(I;W^{1,p}(B))\quad\text{and}\quad\mc{P}([w]^h,Q)\to\mc{P}(w,Q).
$$

Note that if $z=(x,t)\in Q$, then
$$
\begin{aligned}
    [w]^h(z)-w(z) & =\iints{Q_h(0)} (w(z-\sigma)-w(z))\kappa_h(\sigma)\,d\sigma \\
    & = \int_{I_{h^2}(0)}\int_{B_h(0)}(w(x-y,t-\tau)-w(x,t-\tau))\kappa_h(y,\tau)\,dyd\tau \\
    & \hspace{1em} +\int_{I_{h^2}(0)}\int_{B_h(0)}(w(x,t-\tau)-w(x,t))\kappa_h(y,\tau)\,dyd\tau.
\end{aligned}
$$
Applying Minkowski's inequality and Jensen's inequality, we have
$$
\begin{aligned}
    & \|[w]^h(\cdot,t)-w(\cdot,t)\|_{L^2(B)} \\
    & \hspace{2em} \leq\int_{I_{h^2}(0)}\int_{B_h(0)}\|w(\cdot-y,t-\tau)-w(\cdot,t-\tau)\|_{L^2(B)}\kappa_h(y,\tau)\,dyd\tau \\
    & \hspace{3em} +\int_{I_{h^2}(0)}\int_{B_h(0)}\|w(\cdot,t-\tau)-w(\cdot,t)\|_{L^2(B)}\kappa_h(y,\tau)\,dyd\tau \\
    & \hspace{2em} \eqcolon I_1^h(t)+I_2^h(t)
\end{aligned}
$$
for any $t\in I$. 

\textit{Limit for $\sup_{I}I_2^h$}: Since $w\in C_{\loc}(0,T;L_{\loc}^2(\Omega))$, $w:[0,T]\to L^2(B)$ is uniformly continuous in $\overline{I} \subset (0,T)$. Hence, for each fixed $\epsilon>0$, there exists $\delta>0$ such that
$$
\sup_{t\in I}\|w(\cdot,t-\tau)-w(\cdot,t)\|_{L^2(B)}<\epsilon\quad\text{whenever }|\tau|<\delta.
$$
Assume that $0<h<\sqrt{\delta}$. Then we see that for $t\in I$,
$$
\begin{aligned}
    I_2^h(t) & =\int_{I_{h^2}(0)}\int_{B_h(0)}\|w(\cdot,t-\tau)-w(\cdot,t)\|_{L^2(B)}\kappa_h(y,\tau)\,dyd\tau \\
    & \leq \int_{I_{h^2}(0)}\int_{B_h(0)}\sup_{t\in I}\|w(\cdot,t-\tau)-w(\cdot,t)\|_{L^2(B)}\kappa_h(y,\tau)\,dyd\tau \\
    & <\epsilon\int_{I_{h^2}(0)}\int_{B_h(0)}\kappa_h(y,\tau)\,dyd\tau=\epsilon.
\end{aligned}
$$
Therefore,
$$
\lim_{h\to 0}\sup_{t\in I}I_2^h(t)=0.
$$

\textit{Limit for $\sup_{I}I_1^h$}: Let $\epsilon>0$ be given. The Fundamental Theorem of Calculus implies that if $\tau \in I_0$ and $w(\cdot,\tau)\in C^\infty(B_0)$, for every $y\in B_h(0)$,
$$
\begin{aligned}
    \|w(\cdot-y,\tau)-w(\cdot,\tau)\|_{L^2(B)}^2 & = \int_B|w(x-y,\tau)-w(x,\tau)|^2\,dx \\
    & = \int_B\left| \int_0^1\frac{d}{d\lambda}[w(x-\lambda y,\tau)]\,d\lambda \right|^2\,dx \\
    & \leq \int_B\int_0^1\left| -y\cdot Dw(x-\lambda y,\tau) \right|^2\,d\lambda\,dx \\
    & \leq \int_B\int_0^1\left| y \right|^2\left| Dw(x-\lambda y,\tau) \right|^2\,d\lambda\,dx \\
    & \leq \left| y \right|^2\int_0^1\int_B\left| Dw(x-\lambda y,\tau) \right|^2\,dx\,d\lambda \\
    & \leq \left| y \right|^2\int_0^1\int_{B_0}\left| Dw(x,\tau) \right|^2\,dx\,d\lambda \\
    & = |y|^2\|Dw(\cdot,\tau)\|_{L^2(B_0)}^2.
\end{aligned}
$$
By approximation, the same inequality holds for $w(\cdot,\tau)\in W^{1,p}(B_0)$ with $\tau \in I_0$. Therefore, for any $t\in I$,
$$
\begin{aligned}
    I_1^h(t) & =\int_{I_{h^2}(t)}\int_{B_h(0)}\|w(\cdot-y,\tau)-w(\cdot,\tau)\|_{L^2(B)}\kappa_h(y,t-\tau)\,dyd\tau \\
    & \leq\int_{I_{h^2}(t)}\int_{B_h(0)}|y|\|Dw(\cdot,\tau)\|_{L^2(B_0)}\kappa_h(y,t-\tau)\,dyd\tau \\
    & \leq h\int_{I_{h^2}(t)}\int_{B_h(0)}\|Dw(\cdot,\tau)\|_{L^2(B_0)}\kappa_h(y,t-\tau)\,dyd\tau \\
    & = h\int_{I_{h^2}(t)}\|Dw(\cdot,\tau)\|_{L^2(B_0)}\tilde{\kappa}_1^h(t-\tau)\,d\tau,
\end{aligned}
$$
where $\tilde{\kappa}_1^h(\tau)=h^{-2}\tilde{\kappa}_1(\tau/h^2)$. Observe that
$$
\|Dw(\cdot,\tau)\|_{L^2(B_0)}\leq C\|Dw(\cdot,\tau)\|_{L^p(B_0)}
$$
and $0\leq \tilde{\kappa}_1^h\leq ch^{-2}$. Thus,
$$
\begin{aligned}
    & h\int_{I_{h^2}(t)}\|Dw(\cdot,\tau)\|_{L^2(B_0)}\tilde{\kappa}_1^h(t-\tau)\,d\tau \\
    & \quad\leq Ch\int_{I_{h^2}(t)}\|Dw(\cdot,\tau)\|_{L^p(B_0)}\tilde{\kappa}_1^h(t-\tau)\,d\tau \\
    & \quad\leq Ch\left( \int_{I_{h^2}(t)}\|Dw(\cdot,\tau)\|_{L^p(B_0)}^p\,d\tau \right)^{\frac{1}{p}}\left( \int_{I_{h^2}(0)}(\tilde{\kappa}_1^h(\tau))^{\frac{p}{p-1}}\,d\tau \right)^{\frac{p-1}{p}} \\
    & \quad\leq Ch^{1-\frac{2}{p}}\|Dw\|_{L^p(I_0;L^p(B_0))}.
\end{aligned}
$$
Since $Dw\in L^p(I_0;L^p(B_0))$, if $p>2$, we have
$$
\lim_{h\to 0}\sup_{t\in I}I_1^h(t)=0.
$$
On the other hand, if $p=2$, then
$$
I_1^h(t)\leq C\left( \int_{I_{h^2}(t)}\|Dw(\cdot,\tau)\|_{L^2(B_0)}^2\,d\tau \right)^{\frac{1}{2}}.
$$
Define $f(\tau):=\|Dw(\cdot,\tau)\|_{L^2(B_0)}^2$ for $\tau\in I_0$ so that $f\in L^1(I_0)$. By absolute continuity of the Lebesgue integral, for each $\epsilon>0$, there exists $\delta_0>0$ such that
$$
\int_{E}f(\tau)\,d\tau<\epsilon\quad\text{whenever }E\text{ is Borel and }|E|<\delta_0.
$$
Since $|I_{h^2}(t)|=2h^2$, if we take $\delta=\sqrt{\delta_0/2}$ and assume $h<\delta$, then
$$
\int_{I_{h^2}(t)}f(\tau)\,d\tau<\epsilon\quad\text{for any }t\in I.
$$
Therefore, when $p=2$, we get
$$
\lim_{h\to 0}\sup_{t\in I}I_1^h(t)=0.
$$
Consequently, we obtain
$$
[w]^h\to w\quad\text{in }C(I;L^2(B)).
$$
Moreover, since
$$
0\leq\left| \|[w]^h\|_{C(I;L^2(B))}-\|w\|_{C(I;L^2(B))} \right|\leq\|[w]^h-w\|_{C(I;L^2(B))},
$$
we have
$$
\sup_{t \in I} \int_{B} |[w]^h|^2 \, dx = \|[w]^h\|_{C(I;L^2(B))}^2 \to \|w\|_{C(I;L^2(B))}^2 = \sup_{t \in I} \int_{B} |w|^2 \, dx,
$$
and hence
$$
\mc{F}([w]^h,Q)\to\mc{F}(w,Q).
$$
Letting $w_m:=[w]^{h_m}\in C^\infty(Q)$ with $h_m\searrow 0$, we complete the proof.
\end{proof}

\bibliographystyle{abbrv}
\bibliography{ref}{}
\end{document}